\documentclass[12pt]{amsart}
\usepackage{a4wide}
\usepackage{amssymb,bm,mathrsfs}
\usepackage{amsthm}
\usepackage{amsmath}
\usepackage{mathtools}
\usepackage{amscd}
\usepackage{verbatim}
\usepackage[all]{xy}
\usepackage{hyperref}

\usepackage{bookmark}

\usepackage{enumerate}

\numberwithin{equation}{section}

\theoremstyle{plain}
\newtheorem{theorem}{Theorem}[section]
\newtheorem{corollary}[theorem]{Corollary}
\newtheorem{lemma}[theorem]{Lemma}

\theoremstyle{definition}

\theoremstyle{remark}
\newtheorem{remark}[theorem]{Remark}

\newcommand{\C}{\mathbb{C}}
\newcommand{\Q}{\mathbb{Q}}
\newcommand{\R}{\mathbb{R}}
\newcommand{\N}{\mathbb{N}}
\newcommand{\Z}{\mathbb{Z}}
\renewcommand{\H}{\mathbb{H}}
\renewcommand{\O}{\mathcal{O}}

\newcommand{\W}{\mathcal{W}}

\renewcommand{\a}{\mathfrak{a}}

\renewcommand{\d}{\mathfrak{d}}

\newcommand{\ep}{\varepsilon}

\newcommand{\set}[1]{\left\{ #1 \right\}}

\newcommand{\br}[1]{\left( #1 \right)}

\newcommand{\bs}{\backslash}

\newcommand{\se}{\subset}

\newcommand{\defC}{:\;}

\newcommand{\OK}{\O_K}

\newcommand{\IK}{\mathcal{I}_K}

\newcommand{\SLZ}{\SL_2(\Z)}

\newcommand{\GLRp}{\GL_2^+(\R)}

\renewcommand{\matrix}[4]{\begin{pmatrix}#1 & #2\\#3 & #4 \end{pmatrix}}

\newcommand{\abcd}{\matrix{a}{b}{c}{d}}

\newcommand{\matrixCol}[2]{\begin{pmatrix}#1 \\ #2\end{pmatrix}}

\DeclareMathOperator{\SL}{SL}

\DeclareMathOperator{\GL}{GL}

\DeclareMathOperator{\sgn}{sgn}

\newcommand{\und}{\quad\text{and}\quad} 

\newcommand{\ssum}{\sideset{}{'}\sum} 

\renewcommand{\Re}{\operatorname{Re}}

\newcommand{\eg}{\mathfrak{e}_\gamma}
\newcommand{\MpZ}{\operatorname{Mp}_2(\Z)}

\newcommand{\E}{\mathcal{E}}

\begin{document}

\author{Johannes J. Buck}

\address{Fachbereich Mathematik, Technische Universit\"at Darmstadt, Schlossgartenstrasse 7,
D--64289 Darmstadt, Germany}
\email{jbuck@mathematik.tu-darmstadt.de}

\thanks{The author was supported by the DFG Collaborative Research Centre TRR 326 Geometry and Arithmetic of Uniformized Structures, project number 444845124.}

\title[Eisenstein series associated to ideals in real quadratic number fields]{Elliptic Eisenstein series associated to ideals in real quadratic number fields}

\date{\today}

\begin{abstract}
In this paper, we compute for odd fundamental discriminants $D>1$ the Fourier expansion of non-holomorphic elliptic Eisenstein series for $\Gamma_0(D)$ with quadratic nebentypus character $\chi_D$ satisfying a certain plus space condition. For each genus of $\Q(\sqrt{D})$, we obtain an associated plus space condition and corresponding Eisenstein series in all positive even weights. In weight $k=2$, the Fourier coefficients are associated to the geometry of Hirzebruch--Zagier divisors on Hilbert modular surfaces.
\end{abstract}

\maketitle

\tableofcontents


\section{Introduction}

Throughout the paper, $K$ is a real quadratic number field of odd discriminant $D$ with a fixed fractional ideal $\a \se K$.
By $N(\cdot)$ we denote the norm function which is defined on $K$ as well as on the ideal group $\IK$ of $K$.
The quadratic form
\[
q: \a \to \Z, \quad q(x) := \frac{N(x)}{N(\a)}
\]
turns $\a$ into an even lattice of signature $(1,1)$. Its dual is given by $\a \d^{-1}$ with $\d = (\sqrt{D})$ being the different ideal. The discriminant group $\a \d^{-1}/\a$ has order $D$ and is cyclic because $D$ is odd.

In \cite[Section 3]{bruinier2003integrals} the authors associate to each even lattice $L$ of signature $(b^+,b^-)$ and isotropic $\beta \in L'/L$ vector valued real analytic Eisenstein series $E_\beta(\tau,s)$ for $\MpZ$ of weights $2 \le k \in \frac{1}{2} \Z$ such that $2k-b^-+b^+ \equiv 0 \pmod 4$ transforming with the Weil representation (cf.~\cite{kudla2010eisenstein} for an alternative approach).

In this paper, we have a closer look at the Eisenstein series $E_0(z,s)$ for $L = \a$ which we call $\E (z,s)$ from now on.
The condition on the weight translates in this case to $k \in 2\N$. Since the weight $k$ is integral, $\E (z,s)$ is a vector valued modular form for $\SLZ$.

From \cite[Proposition 4.5]{scheithauer2009weil} it follows that the zero component of a vector valued modular form for $\SLZ$ is a scalar valued modular form with respect to a character depending on the discriminant group of the lattice. In our case it turns out that independent of the ideal $\a \in \IK$ the character associated to the discriminant group $\a\d^{-1}/\a$ is given by the Kronecker symbol character $\chi_D(n) :=\br{\frac{D}{n}}$, even though on a fixed real quadratic field $K$ there may occur non-isomorphic discriminant groups for different ideals. Therefore, $\E_0(\cdot,s) \in A_k(D,\chi_D)$ where $A_k(D,\chi_D)$ denotes the vector space of scalar valued non-holomorphic modular forms of weight $k$, level $\Gamma_0(D)$ and nebentypus $\chi_D$.

Recall the definition of the Petersson slash operator
\[
(f \mid_k M) (z) = \det(M)^{k/2} (cz +d)^{-k} f(Mz)
\quad\text{for}\quad
M = \abcd \in \GLRp,
\]
which defines a group action of $\GLRp$ on the set of all scalar valued functions $f : \H \to \C$.
The Fricke involution
\[
A_k(D,\chi_D) \to A_k(D,\chi_D),\quad f \mapsto f \mid_k W_D, \quad W_D:=\matrix{0}{-1}{D}{0}
\]
defines an automorphism on $A_k(D,\chi_D)$. Thus, we have $(\E_0(\cdot,s) \mid_k W_D) \in A_k(D,\chi_D)$. We define
\[
E(\tau,s) := \frac{D^{(1-k)/2}}{2} (\E_0(\cdot,s) \mid_k W_D)(\tau)
\]
and make the Fourier expansion of $E(\tau,s)$ explicit in the course of this paper (cf.~Theorem~\ref{s-coef}).
We show that the generalized divisor sum $\sigma(\a,m,s)$ discussed in \cite{buckRepNumbers} occurs as factor in the Fourier coefficients of $E(\tau,s)$. This shows that $E(\tau,s)$ satisfies a certain plus space condition depending on the genus of $\a$ (cf.~Remark~\ref{plus-space}).

Subsequently, we determine the Fourier expansion of the holomorphic Eisenstein series $E(\tau,0)$ in Theorem~\ref{eisen-s0} and compute the derivative $E'(\tau,0)$ in $s$ at $s=0$ in Theorem~\ref{eisen-der-s0}.
Finally, we have a closer look at the weight $k=2$ where the occuring values of the generalized divisor sum represent the volumes of Hirzebruch--Zagier divisors living on the Hilbert modular surface $X(\a) := \SL(\OK \oplus \a) \bs \H^2$ (cf.~Remark~\ref{geometric-interpretation}).

\section*{Acknowledgement}
The author thanks Jan Bruinier for his support and valuable comments.

\section{The Eisenstein series in dependence of \texorpdfstring{$s$}{s}} 

For the complex variable $z$ we use $x$ and $y$ for its real and imaginary part (and for $\tau$ we use $u$ and $v$, respectively).
By \cite[Proposition 3.1]{bruinier2003integrals} the vector valued Eisenstein series
\[
\E(z, s) := \frac{1}{2} \sum_{ (M,\phi) \in \tilde\Gamma_\infty \bs \MpZ } (\mathfrak{e}_0y^s) \mid_k^* (M,\phi)
\]
of weight $k \in 2 \N$ associated to the lattice $\a$\footnote{The dependence on $\a$ is hidden in the definition of the vector valued Petersson slash operator $\mid_k^*$. In its definition the Weil representation associated to the discrimiant group $\a\d^{-1}/\a$ comes into play (cf.~\cite[Section~2]{bruinier2003integrals} for details).} has the Fourier expansion
\[
\E(z, s) = \sum_{\gamma \in \a\d^{-1}/\a} \sum_{n \in \Z - q(\gamma)} c_0 (\gamma, n,s,y) \eg(nx)
\]
where the coefficients $c_0 (\gamma, n,s,y)$ are given by
\[
\begin{dcases}
2 \delta_{0,\gamma} y^s + 2 \pi y^{1-k-s} \frac{\Gamma(k+2s-1)}{\Gamma(k+s)\Gamma(s)} \ssum_{c \in \Z} |2c|^{1-k-2s} H_c^*(0,0,\gamma,0),\quad &n=0,\\
\frac{2^k \pi^{s+k } |n|^{s+k-1} }{\Gamma(s+k)} \W_s(4 \pi n y) \ssum_{c \in \Z} |c|^{1-k-2s} H_c^*(0,0,\gamma,n),\quad &n>0,\\
\frac{2^k \pi^{s+k } |n|^{s+k-1} }{\Gamma(s)} \W_s(4 \pi n y) \ssum_{c \in \Z} |c|^{1-k-2s} H_c^*(0,0,\gamma,n),\quad &n<0.
\end{dcases}
\]
Here $H^*_c(\beta,m,\gamma,n)$ denotes a certain Kloosterman sum (cf.~\cite[eq.~(3.4)]{bruinier2003integrals}) and
\begin{align} \label{Ws-def} 
\mathcal W_s(v) := |v|^{-k/2} W_{\sgn(v)k/2,(1-k)/2-s}(|v|)
\end{align}
is for $s \in \C$ and $v \in \R^\times$ as in \cite[eq.~(3.2)]{bruinier2003integrals} derived from the usual $W$-Whittaker function $W_{\nu,\mu}(z)$.
In \cite{bruinier2003integrals} in the displayed equation after eq.~(3.6) the relation
\[
\ssum_{c \in \Z} |c|^{1-k-2s} H_c^*(0,0,\gamma,n)
= \frac{2 i^k}{\sqrt{D} \zeta(2s+k-1)} \sum_{a=1}^\infty N_{\gamma,n}(b)b^{-k-2s}
\]
is inferred.
Here, the representation number $N_{\gamma,n}(a)$ is defined by
\begin{align} \label{Ngamma-def} 
N_{\gamma,n}(b) := \# \set{ x \in \a/b\a \defC \frac{N(x-\gamma)}{N(\a)} + n  \equiv 0 \pmod b }.
\end{align}
This allows us to rewrite the Fourier coefficients of $\E(z, s)$ to
\[
c_0 (\gamma,n,s,y) = 
\begin{dcases}
2 \delta_{0,\gamma} y^s +  \frac{i^k 2^{3-k-2s} \pi y^{1-k-s} }{\sqrt{D} \zeta(2s+k-1)} \frac{\Gamma(k+2s-1)}{\Gamma(k+s)\Gamma(s)}  \sum_{a=1}^\infty N_{\gamma,0}(b)b^{-k-2s},\, &n=0,\\
\frac{i^k 2^{k+1} \pi^{s+k} |n|^{s+k-1}}{\sqrt{D} \Gamma(s+k)\zeta(2s+k-1)} \W_s(4 \pi n y) \sum_{a=1}^\infty N_{\gamma,n}(b)b^{-k-2s}, &n>0,\\
\frac{i^k 2^{k+1} \pi^{s+k} |n|^{s+k-1}}{\sqrt{D} \Gamma(s)\zeta(2s+k-1)} \W_s(4 \pi n y) \sum_{a=1}^\infty N_{\gamma,n}(b)b^{-k-2s}, &n<0.
\end{dcases}
\]
Now recall that we are actually interested in the Fourier expansion of
\[
E(\tau,s) = \frac{D^{(1-k)/2}}{2} (\E_0(\cdot,s) \mid_k W_D)(\tau).
\]
In the following we make use of the decomposition
\[
W_D = S \cdot V_D
\quad\text{with}\quad
S:=\matrix{0}{-1}{1}{0}
\und
V_D:=\matrix{D}{0}{0}{1}.
\]
Since $\E(z,s)$ is invariant under the $\SLZ$ action of the Petersson slash operator of weight~$k$ with respect to the Weil representation, we obtain from applying it to the matrix $S$ the identity
\[
\E_0(z,s) = \frac{z^{-k}}{\sqrt{D}} \sum_{\gamma \in \a\d^{-1}/\a} \E_\gamma(-1/z,s).
\]
Thus, it follows
\[
(\E_0(\cdot, s) \mid_k S)(\tau) = \tau^{-k} \E_0(-1/\tau, s) = \frac{1}{\sqrt{D}} \sum_{\gamma \in \a\d^{-1}/\a} \E_\gamma(\tau,s).
\]
Therefore,
\begin{align*}
E(\tau,s)
&= \frac{D^{(1-k)/2}}{2 \sqrt{D}} \sum_{\gamma \in \a\d^{-1}/\a} (\E_\gamma(\cdot,s) \mid_k V_D)(\tau)\\
&= \frac{1}{2} \sum_{\gamma \in \a\d^{-1}/\a} \E_\gamma(D \tau,s)
= \frac{1}{2} \sum_{\gamma \in \a\d^{-1}/\a}   \sum_{n \in \Z - q(\gamma)} c_0 (\gamma, n,s,D v) e(nDu).
\end{align*}
Hence, the Fourier coefficients $c(m,s,v)$ of
\[
E(\tau,s) = \sum_{m \in \Z} c(m,s,v) e(mu)
\]
are given by
\[
c(m,s,v) = \frac{1}{2} \sum_{ \substack{\gamma \in \a\d^{-1}/\a \\ m/D + q(\gamma) \in \Z }} c_0 (\gamma, m/D,s,D v).
\]
Note for that reformulation that $Dq(\gamma) \in \Z$ for all $\gamma \in \a\d^{-1}$. Therefore, the Fourier coefficients of $E(\tau,s)$ are indexed by $m \in \Z$.
An important insight following directly from the definition of $N_{\gamma,n}(b)$ (cf.~equation~\eqref{Ngamma-def}) for making the coefficients $c(m,s,v)$ more explicit is
\begin{align*}
\sum_{ \substack{\gamma \in \a\d^{-1}/\a  \\ n + N(\gamma)/N(\a) \in \Z }} N_{\gamma,n}(b)
&= \# \set{ x \in \a\d^{-1}/b\a \defC \frac{N(x)}{N(\a)} + n \equiv 0 \pmod b  }.
\end{align*}
The representation number on the right hand side of this equation plays an important role in \cite{buckRepNumbers} where it is denoted by $G^b(\a,m,0)$ with $m := nD$.
We obtain
\[
c(0,s,v) = D^s v^s +  \frac{i^k 2^{2-k-2s} \pi v^{1-k-s} }{D^{s+k-1/2} \zeta(2s+k-1)} \frac{\Gamma(k+2s-1)}{\Gamma(k+s)\Gamma(s)}  \sum_{a=1}^\infty G^b(\a,0,0)b^{-k-2s}
\]
and
\[
c(m,s,v) = 
\begin{dcases}
\frac{i^k 2^{k} \pi^{s+k} |m|^{s+k-1}}{D^{s+k-1/2} \Gamma(s+k)\zeta(2s+k-1)} \W_s(4 \pi m v) \sum_{a=1}^\infty G^b(\a,m,0) b^{-k-2s}, &m>0,\\
\frac{i^k 2^{k} \pi^{s+k} |m|^{s+k-1}}{D^{s+k-1/2} \Gamma(s)\zeta(2s+k-1)} \W_s(4 \pi m v) \sum_{a=1}^\infty G^b(\a,m,0) b^{-k-2s}, &m<0.
\end{dcases}
\]
Now we can make use of the main result of \cite{buckRepNumbers}, the representation
\[
\sum_{b=1}^\infty G^b(\a,m,0)b^{-s} =|m|^{-s/2}  \frac{\zeta(s-1)}{L(s,\chi_D)} \sigma(\a,m,1-s).
\]
For $m>0$ we obtain
\begin{align*}
c(m,s,v)
&= \frac{i^k 2^{k} \pi^{s+k} |m|^{s+k-1}}{D^{s+k-1/2} \Gamma(s+k)\zeta(2s+k-1)} \W_s(4 \pi m v) \sum_{a=1}^\infty G^b(\a,m,0) b^{-k-2s}\\
&= \frac{i^k 2^{k} \pi^{s+k} |m|^{s+k-1}}{D^{s+k-1/2} \Gamma(s+k)\zeta(2s+k-1)} \W_s(4 \pi m v) |m|^{-(k+2s)/2}  \frac{\zeta(2s+k-1)}{L(2s+k,\chi_D)} \\
&\quad \times \sigma(\a,m,1-2s-k)\\
&= \frac{i^k 2^{k} \pi^{s+k} |m|^{k/2-1} \W_s(4 \pi m v) }{D^{s+k-1/2} \Gamma(s+k)L(2s+k,\chi_D)}  \sigma(\a,m,1-2s-k).
\end{align*}
Using the functional equation
\[
L(2s+k,\chi_D) = \br{\frac{\pi}{D}}^{2s+k-1/2}  \frac{\Gamma(1/2-s-k/2)}{\Gamma(s+k/2)}  L(1-2s-k,\chi_D),
\]
we see
\begin{align*}
c(m,s,v)
&= \frac{i^k 2^{k} \pi^{s+k} |m|^{k/2-1} \W_s(4 \pi m v) }{D^{s+k-1/2} \Gamma(s+k) L(1-2s-k,\chi_D)} \br{\frac{D}{\pi}}^{2s+k-1/2}  \frac{\Gamma(s+k/2)}{\Gamma(1/2-s-k/2)} \\
&\quad \times  \sigma(\a,m,1-2s-k)\\
&= \frac{i^k 2^{k} \pi^{1/2-s} D^s |m|^{k/2-1} \W_s(4 \pi m v) }{ \Gamma(s+k) L(1-2s-k,\chi_D)}  \frac{\Gamma(s+k/2)}{\Gamma(1/2-s-k/2)}  \sigma(\a,m,1-2s-k)\\
&= \frac{i^k 2^{k} \pi^{1/2-s} D^s  \Gamma(s+k/2)}{L(1-2s-k,\chi_D) \Gamma(s+k) \Gamma(1/2-s-k/2)} |m|^{k/2-1} \sigma(\a,m,1-2s-k)\\
&\quad \times \W_s(4 \pi m v).
\end{align*}
Using Legendre's duplication formula, we obtain
\[
\Gamma(s+k/2) = \frac{\sqrt{\pi} \Gamma(2s+k)}{2^{2s+k-1} \Gamma(s+k/2+1/2)}.
\]
Now, Euler's reflection formula yields
\[
\Gamma(s+k/2+1/2) \Gamma(1/2-s-k/2) = \frac{\pi}{ \sin\br{ \pi (s+k/2+1/2) } } = \frac{\pi i^k}{ \cos(\pi s) }.
\]
Combined, we get
\[
\Gamma(s+k/2)
= \frac{ \Gamma(2s+k) \cos(\pi s) \Gamma(1/2-s-k/2) }{2^{2s+k-1} \sqrt{\pi} i^k}.
\]
Inserted into $c(m,s,v)$ we obtain
\begin{align*}
c(m,s,v)
&= \frac{ 2^{1-2s} \pi^{-s} D^s  \Gamma(2s+k) \cos(\pi s)   }{L(1-2s-k,\chi_D) \Gamma(s+k)    } |m|^{k/2-1} \sigma(\a,m,1-2s-k)\W_s(4 \pi m v)\\
&= 2 \br{\frac{D}{4\pi}}^s  \frac{  \cos(\pi s)   \Gamma(2s+k)   }{\Gamma(s+k)  L(1-2s-k,\chi_D) } |m|^{k/2-1} \sigma(\a,m,1-2s-k)\W_s(4 \pi m v).
\end{align*}
Analogously, we get for $m<0$
\[
c(m,s,v) = 2 \br{\frac{D}{4\pi}}^s  \frac{  \cos(\pi s)   \Gamma(2s+k)   }{\Gamma(s)  L(1-2s-k,\chi_D) } |m|^{k/2-1} \sigma(\a,m,1-2s-k)\W_s(4 \pi m v).
\]
For $m=0$ we use
\[
\sum_{b=1}^\infty G^b(\a,0,0)b^{-s} = \zeta(s-1) \frac{L(s-1,\chi_D)}{L(s,\chi_D)}
\]
from \cite[Corollary 6.2]{buckRepNumbers} to obtain
\begin{align*}
&\frac{i^k 2^{2-k-2s} \pi v^{1-k-s} }{D^{s+k-1/2} \zeta(2s+k-1)} \frac{\Gamma(k+2s-1)}{\Gamma(k+s)\Gamma(s)}  \sum_{a=1}^\infty G^b(\a,0,0)b^{-k-2s}\\
=\ &\frac{i^k 2^{2-k-2s} \pi v^{1-k-s} }{D^{s+k-1/2} \zeta(2s+k-1)} \frac{\Gamma(k+2s-1)}{\Gamma(k+s)\Gamma(s)}  \zeta(k+2s-1) \frac{L(k+2s-1,\chi_D)}{L(k+2s,\chi_D)}\\
=\ &\frac{i^k 2^{2-k-2s} \pi v^{1-k-s} }{D^{s+k-1/2}} \frac{\Gamma(k+2s-1)}{\Gamma(k+s)\Gamma(s)}  \frac{L(k+2s-1,\chi_D)}{L(k+2s,\chi_D)}.
\end{align*}
Altogether we have proven
\begin{theorem} \label{s-coef} 
The Fourier coefficients $c(m,s,v)$ of the Eisenstein series
\[
E(\tau,s) = \sum_{m \in \Z} c(m,s,v) e(mu)
\]
are given by
\[
\begin{dcases}
D^s v^s +  \frac{i^k 2^{2-k-2s} \pi v^{1-k-s} }{D^{s+k-1/2}} \frac{\Gamma(k+2s-1)}{\Gamma(k+s)\Gamma(s)}  \frac{L(k+2s-1,\chi_D)}{L(k+2s,\chi_D)}, &m=0,\\
2 \br{\frac{D}{4\pi}}^s  \frac{  \cos(\pi s)   \Gamma(2s+k)   }{\Gamma(s+k)  L(1-2s-k,\chi_D) } |m|^{k/2-1} \sigma(\a,m,1-2s-k)\W_s(4 \pi m v), &m>0,\\
2 \br{\frac{D}{4\pi}}^s  \frac{  \cos(\pi s)   \Gamma(2s+k)   }{\Gamma(s)  L(1-2s-k,\chi_D) } |m|^{k/2-1} \sigma(\a,m,1-2s-k)\W_s(4 \pi m v), &m<0.
\end{dcases}
\]
\end{theorem}

\begin{remark} \label{plus-space} 
In \cite{buckRepNumbers} it is shown that the value of $\sigma(\a,m,s)$ depends not on the ideal $\a \in \IK$ itself but only on its genus.\footnote{Note that instead of using \cite{buckRepNumbers} one could use that by definition the Eisenstein series $\E(z,s)$ (and therefore $E(\tau,s)$ as well) depends on the discriminant group $\a\d^{-1}/\a$ of $\a$ only and not on the ideal itself. However, the discriminant groups of ideals from the same genus are isomorphic.}  With Theorem~\ref{s-coef} we now see that the Eisenstein series $E(\tau,s)$ as well depends not on the ideal $\a \in \IK$ itself but only on its genus. Therefore, we may assume that $\a$ is integral and coprime to $D$ and obtain with \cite[Lemma~7.4]{buckRepNumbers} that all Fourier coefficients vanish identically in $s$ whose index $m$ satisfies $\chi_{D(p)}(N(\a)m)=-1$ for a prime $p \mid D$. Hence, $E(\tau,s)$ lies in a certain plus space determined by the genus of $\a$.
\end{remark}

\begin{remark}
Note that the formulae of Theorem~\ref{s-coef} are a generalization of the formulae in \cite[Theorem~2.2]{bruinier2006cm} to arbitrary odd fundamental discriminants $D$ (\cite{bruinier2006cm} considers the case where $D$ is prime). However, the approach presented here is different from the approach of \cite{bruinier2006cm}. There, the authors did not start with a vector valued Eisenstein series in order to construct their so-called Eisenstein series $E_k^+(\tau,s)$. Instead, they determined the right linear conbination of the already known Eisenstein series $E_k^\infty(\tau,s)$ and $E_k^0(\tau,s)$ associated to the two cusps $0$ and $\infty$ of $\Gamma_0(D)$ in order to match the plus space condition. Since for prime $D$ there exists only one genus, in \cite{bruinier2006cm} there is no need to introduce an ideal $\a \in \IK$ representing the genus.

Further note that the divisor sum in \cite{bruinier2006cm} depends on the weight $k$ whereas our divisor sum is independent of $k$. This explains our additional factor $|m|^{k/2-1}$. Beyond that, in \cite{bruinier2006cm} there is no need for a definition of the divisor sum $\sigma(\a,m,s)$ for negative $m$ since it agrees with $\sigma(\a,|m|,s)$ because $D$ is prime. In our more general situation, however, there are cases where $\sigma(\a,m,s) \ne \sigma(\a,-m,s)$.
\end{remark}

\section{The Eisenstein series and its derivative at \texorpdfstring{$s=0$}{s=0}} 

With Theorem~\ref{s-coef} it is easy to prove the next theorem.
\begin{theorem} \label{eisen-s0} 
The Eisenstein series $E(\tau,0)$ is holomorphic. It has the Fourier expansion
\[
E(\tau,0) = 1 + \frac{2}{ L(1-k,\chi_D)} \sum_{m=1}^\infty |m|^{k/2-1} \sigma(\a,m,1-k) e(m\tau).
\]
\end{theorem}
\begin{proof}
This theorem is proved by plugging in $s=0$ into Theorem~\ref{s-coef}. Because of $\W_0(v) = e^{-v/2}$ we have
\[
\W_0(4 \pi m v) e(mu) = e^{-2 \pi m v} e^{2 \pi i m u} = e^{2 \pi i m \tau } = e(m\tau).
\]
We lose the non-holomorphic components because of the simple pole at $s=0$ of $\Gamma(s)$.
\end{proof}

Next, we want to determine the derivative
\[
E'(\tau,0) := \frac{d}{ds} E(\tau,s) \mid_{s=0}
\]
which turns out to be not holomorphic anymore.
For that purpose the next lemma is useful.

\begin{lemma} \label{W-derivative} 
We have for $v>0$
\begin{align*}
\W_0'(v) := \frac{d}{ds} \W_s'(v) \mid_{s=0} = e^{-v/2} \sum_{j=1}^{k-1} \matrixCol{k-1}{j} \frac{(j-1)!}{v^j}.
\end{align*}
\end{lemma}
\begin{proof}
From equation~\eqref{Ws-def} it follows for $v>0$
\[
\W_s(v) = v^{-k/2} W_{k/2,(1-k)/2-s}(v).
\]
Using \cite[13.4.3]{handbook}, we obtain
\[
W_{k/2,(1-k)/2-s}(v)
= e^{-v/2} v^{1-k/2-s} U (1-k-s,2-k-2s,v)
\]
and hence
\[
\W_s(v) = e^{-v/2} v^{1-k-s} U (1-k-s,2-k-2s,v).
\]
Using Kummer's transformation (cf.~\cite[13.2.40]{handbook}) together with \cite[13.4.4]{handbook}, we end up in
\begin{align*}
\mathcal W_s(v)
= e^{-v/2} v^{s} U (s,2s+k,v)
= e^{-v/2} \frac{v^{s}}{\Gamma(s)} \int_0^\infty e^{-vt} t^{s-1} (1+t)^{s+k-1} dt
\end{align*}
for $\Re(s)>0$. Using
\[
\int_0^\infty e^{-vt} t^{s-1} dt
= \frac{1}{v} \int_0^\infty e^{-t} \br{ \frac{t}{v} }^{s-1} dt
= \frac{1}{v^s} \int_0^\infty e^{-t} t^{s-1} dt
= \frac{\Gamma(s)}{v^s},
\]
we can rewrite
\[
\mathcal W_s(v) = e^{-v/2} \br{1 + \frac{v^{s}}{\Gamma(s)} \int_0^\infty e^{-vt} t^{s-1} \br{ (1+t)^{s+k-1} -1}dt }
\]
where the integral on the right hand side converges now for $\Re(s)>-1$.
Having $\Gamma(s)$ with the simple pole of residue $1$ at $s=0$ in the denominator makes it easy to determine $\mathcal W_0'(v)$ by simply plugging in $s=0$ in the remaining factors and omitting $\Gamma(s)$:
\begin{align*}
\mathcal W_0'(v)
&= e^{-v/2} \int_0^\infty e^{-vt} \br{ (1+t)^{k-1} -1} \frac{dt}{t}\\
&= e^{-v/2} \int_0^\infty e^{-vt} \sum_{j=1}^{k-1} \matrixCol{k-1}{j} t^{j} \frac{dt}{t}\\
&= e^{-v/2} \sum_{j=1}^{k-1} \matrixCol{k-1}{j} \frac{1}{v} \int_0^\infty e^{-t} \br{\frac{t}{v}}^{j-1} dt\\
&= e^{-v/2} \sum_{j=1}^{k-1} \matrixCol{k-1}{j} \frac{(j-1)!}{v^j}.
\end{align*}
\end{proof}

\begin{lemma} \label{Gamma-quotient} 
Let $n \in \N$. Then it holds
\[
\frac{\Gamma'(n)}{\Gamma(n)} = H_{n-1} - \gamma
\]
where $H_{n}$ is the $n$th harmonic number and $\gamma$ is the Euler--Mascheroni constant.
\end{lemma}
\begin{proof}
By the functional equation of the Gamma function we have
\[
\frac{\Gamma'(s+1)}{\Gamma(s+1)} = \frac{s \Gamma'(s) + \Gamma(s)}{s\Gamma(s)} = \frac{\Gamma'(s)}{\Gamma(s)} + \frac{1}{s}.
\]
The statement follows now by induction and $\Gamma'(1)=-\gamma$.
\end{proof}

\begin{theorem} \label{eisen-der-s0} 
The Fourier coefficients of
\[
E'(\tau,0) = \sum_{m \in \Z} \tilde c(m,v) e(m\tau)
\]
are given by
\[
\tilde c(m,v) =
\begin{dcases}
\log(vD) + \frac{i^k 2^{2-k} \pi v^{1-k} }{(k-1)D^{k-1/2}}  \frac{L(k-1,\chi_D)}{L(k,\chi_D)}, \quad &m=0,\\
\frac{2(k-1)!}{ L(1-k,\chi_D) } |m|^{k/2-1} \sigma(\a,m,1-k) \Gamma(1-k,4 \pi |m| v), \quad &m<0
\end{dcases}
\]
for non-positive $m$. For $m>0$ we have
\begin{align*}
&\tilde c(m,v) = \frac{2 |m|^{k/2-1} \sigma(\a,m,1-k)  }{L(1-k,\chi_D)}\\
&\times \br{\log \br{ \frac{D}{4\pi}} +  H_{k-1}-\gamma +2 \frac{  L'(1-k,\chi_D) }{ L(1-k,\chi_D)}
-2 \frac{\sigma'(\a,m,1-k)}{\sigma(\a,m,1-k)} +\sum_{j=1}^{k-1} \matrixCol{k-1}{j} \frac{(j-1)!}{(4 \pi m v)^j}}.
\end{align*}
\end{theorem}
\begin{proof}
Again we work with Theorem~\ref{s-coef}. Computing the derivative at those components where we have $\Gamma(s)$ in the denominator is easy. We simply ignore $\Gamma(s)$ and plug in $s=0$ in the remaining factors like in the proof of Lemma~\ref{W-derivative}. Hence, we obtain
\begin{align*}
\tilde c(0,v)
&= \log(vD) + 
\frac{i^k 2^{2-k} \pi v^{1-k} }{D^{k-1/2}} \frac{\Gamma(k-1)}{\Gamma(k)}  \frac{L(k-1,\chi_D)}{L(k,\chi_D)}\\
&= \log(vD) + \frac{i^k 2^{2-k} \pi v^{1-k} }{(k-1)D^{k-1/2}}  \frac{L(k-1,\chi_D)}{L(k,\chi_D)}.
\end{align*}
Using $\mathcal W_0(v) = e^{-v/2} \Gamma(1-k,|v|)$ for $v<0$ we obtain for $m<0$ with the same proceeding
\begin{align*}
c'(m,0,v)
&= 2 \frac{  \Gamma(k)   }{ L(1-k,\chi_D) } |m|^{k/2-1} \sigma(\a,m,1-k)\W_0(4 \pi m v)\\
&= \frac{2(k-1)!}{ L(1-k,\chi_D) } |m|^{k/2-1} \sigma(\a,m,1-k) e^{-2 \pi m v} \Gamma(1-k,4 \pi |m| v) .
\end{align*}
Thus
\[
\tilde c(m,v) = \frac{2(k-1)!}{ L(1-k,\chi_D) } |m|^{k/2-1} \sigma(\a,m,1-k) \Gamma(1-k,4 \pi |m| v).
\]
Finally, for $m>0$ we need Lemma~\ref{W-derivative} and we actually have to compute derivaties. We compute the logarithmic derivative at $s=0$ first:
\begin{align*}
\frac{c'(m,0,v)}{c(m,0,v)}
= \log \br{ \frac{D}{4\pi}} + \frac {- \pi \sin(\pi \cdot 0)}{ \cos(\pi \cdot 0)} + \frac{2 \Gamma'(k)}{\Gamma(k)} - \frac{\Gamma'(k)}{\Gamma(k)} - \frac{ -2 L'(1-k,\chi_D) }{ L(1-k,\chi_D)}\\
+ \frac{-2\sigma'(\a,m,1-k)}{\sigma(\a,m,1-k)} +\sum_{j=1}^{k-1} \matrixCol{k-1}{j} \frac{(j-1)!}{(4 \pi m v)^j}\\
= \log \br{ \frac{D}{4\pi}} +  \frac{\Gamma'(k)}{\Gamma(k)} +2 \frac{  L'(1-k,\chi_D) }{ L(1-k,\chi_D)}
-2 \frac{\sigma'(\a,m,1-k)}{\sigma(\a,m,1-k)} +\sum_{j=1}^{k-1} \matrixCol{k-1}{j} \frac{(j-1)!}{(4 \pi m v)^j}.
\end{align*}
Further, we have
\[
c(m,0,v) = \frac{2 |m|^{k/2-1}}{L(1-k,\chi_D)} \sigma(\a,m,1-k) \W_0(4 \pi m v).
\]
Therefore in total, we obtain with Lemma~\ref{Gamma-quotient}
\begin{align*}
&\tilde c(m,v) = \frac{2 |m|^{k/2-1} \sigma(\a,m,1-k)  }{L(1-k,\chi_D)}\\
&\times \br{\log \br{ \frac{D}{4\pi}} +  H_{k-1}-\gamma +2 \frac{  L'(1-k,\chi_D) }{ L(1-k,\chi_D)}
-2 \frac{\sigma'(\a,m,1-k)}{\sigma(\a,m,1-k)} +\sum_{j=1}^{k-1} \matrixCol{k-1}{j} \frac{(j-1)!}{(4 \pi m v)^j}}.
\end{align*}
\end{proof}

\section{The Eisenstein series and its derivative at \texorpdfstring{$s=0$}{s=0} for weight \texorpdfstring{$k=2$}{k=2}} 

The most important special case of Theorem~\ref{eisen-s0} is
\begin{corollary} \label{k2-holo} 
For $k=2$ we have
\[
E(\tau,0) = 1 + \frac{2}{ L(-1,\chi_D)} \sum_{m=1}^\infty \sigma(\a,m,-1) e(m\tau).
\]
\end{corollary}
\begin{remark} \label{geometric-interpretation} 
This particular Eisenstein series has a geometric interpretation in terms of Hilbert modular surfaces (cf.~\cite[Theorem 5.1]{geer1988hms}, \cite{hirzebruch1976intersection}): With a suitable scaling of that series the constant term can be interpreted as the volume of the Hilbert modular surface $X(\a)$ and the coefficients for $m>0$ as the volumes of the respective Hirzebruch--Zagier divisors.
\end{remark}

Finally, let us have a look at $E'(\tau,0)$ for $k=2$ as well.
In the constant term of its Fourier expansion the value $L(1,\chi_D)$ occurs which is related by the so-called class number formula
\[
h_K = \frac{\sqrt{D}}{2\log(\ep_0)}L(1,\chi_D)
\]
to the class number $h_K$ of the number field $K$. Here $\ep_0$ is the fundamental unit of $\OK$, the ring of integers of $K$, i.e., it is the smallest $\ep \in \OK^\times$ with $\ep>1$.
We obtain our final result:
\begin{corollary} \label{k2-der} 
The Fourier coefficients of
\[
E'(\tau,0) = \sum_{m \in \Z} \tilde c(m,v) e(m\tau)
\]
for weight $k=2$ are given by
\[
\tilde c(m,v) =
\begin{dcases}
\log(vD) + \frac{h_K \log(\ep_0)}{\sqrt{D}v \pi L(-1,\chi_D)}, \quad &m=0,\\
\frac{2 \sigma(\a,m,-1)}{ L(-1,\chi_D) }  \Gamma(-1,4 \pi |m| v), \quad &m<0
\end{dcases}
\]
for non-positive $m$. For $m>0$ we have
\begin{align*}
&\tilde c(m,v) = \frac{2 \sigma(\a,m,-1)  }{L(-1,\chi_D)}
 \br{\log \br{ \frac{D}{4\pi}} +  1-\gamma +2 \frac{  L'(-1,\chi_D) }{ L(-1,\chi_D)}
-2 \frac{\sigma'(\a,m,-1)}{\sigma(\a,m,-1)} +\frac{1}{4 \pi m v}}.
\end{align*}
\end{corollary}
\begin{proof}
We plug in $k=2$ into Theorem \ref{eisen-der-s0} and obtain
\[
\tilde c(m,v) =
\begin{dcases}
\log(vD) - \frac{\pi v^{-1} }{D^{3/2}}  \frac{L(1,\chi_D)}{L(2,\chi_D)}, \quad &m=0,\\
\frac{2 \sigma(\a,m,-1)}{ L(-1,\chi_D) }  \Gamma(-1,4 \pi |m| v), \quad &m<0.
\end{dcases}
\]
The functional equation of $L(s,\chi_D)$ yields
\[
L(2,\chi_D)
= -\frac{2\pi^2}{D^{3/2}}L(-1,\chi_D).
\]
Together with the class number formula we obtain
\begin{align*}
-\frac{\pi v^{-1} }{D^{3/2}}  \frac{L(1,\chi_D)}{L(2,\chi_D)}
= -\frac{\pi v^{-1} }{D^{3/2}}  \frac{2 h_K \log(\ep_0)}{\sqrt{D}}\frac{-D^{3/2}}{2\pi^2 L(-1,\chi_D)}
= \frac{h_K \log(\ep_0)}{\sqrt{D}v \pi L(-1,\chi_D)}.
\end{align*}
This finishes the proof for non-positive $m$. For $m>0$ we plug in $k=2$ into Theorem~\ref{eisen-der-s0} as well and obtain
\begin{align*}
&\tilde c(m,v) = \frac{2 \sigma(\a,m,-1)  }{L(-1,\chi_D)}
 \br{\log \br{ \frac{D}{4\pi}} +  1-\gamma +2 \frac{  L'(-1,\chi_D) }{ L(-1,\chi_D)}
-2 \frac{\sigma'(\a,m,-1)}{\sigma(\a,m,-1)} +\frac{1}{4 \pi m v}}
\end{align*}
as desired.
\end{proof}


\end{document}